\documentclass[11pt,leqno,twoside]{article}

\usepackage{geometry}
\usepackage{amsmath}
\usepackage{amsfonts}
\usepackage{amsthm}
\usepackage{indentfirst}

\theoremstyle{plain}
\newtheorem{theorem}{\indent\rm T\,h\,e\,o\,r\,e\,m\;}[section]
\newtheorem{lemma}{\indent\rm L\,e\,m\,m\,a\;}[section]
\newtheorem{proposition}{\indent\rm P\,r\,o\,p\,o\,s\,i\,t\,i\,o\,n\;}[section]

\theoremstyle{definition}
\newtheorem{definition}{\indent\rm D\,e\,f\,i\,n\,i\,t\,i\,o\,n\;}[section]

\theoremstyle{remark}
\newtheorem{remark}{\indent\rm R\,e\,m\,a\,r\,k\;}[section]

\renewenvironment{proof}{\indent\rm P\,r\,o\,o\,f.\;}{\hfill $\square$ \\ \indent}

\makeatletter                                                  %
\renewcommand*{\@seccntformat}[1]{
  \csname the#1\endcsname\;-                                   %
}                                                              %
\renewcommand{\section}{\@startsection{section}{1}{0mm}        %
   {1.5\baselineskip}
   {1\baselineskip}
   {\indent\normalfont\normalsize\bfseries}
   }                                                           %
\renewcommand*{\@seccntformat}[1]{
  \normalfont\bfseries\csname the#1\endcsname\;-               %
}                                                              %
\renewcommand\subsection{\@startsection                        %
  {subsection}{2}{0mm}
  {1.5\baselineskip}
  {1\baselineskip}
  {\indent\normalfont\normalsize\itshape}}
\renewcommand*{\@seccntformat}[1]{
  \normalfont\bfseries\csname the#1\endcsname\;-               %
}                                                              %
\renewcommand\subsubsection{\@startsection                     %
  {subsubsection}{2}{0mm}
  {1.5\baselineskip}
  {1\baselineskip}
  {\indent\normalfont\normalsize\texttt}}
\makeatother                                                   %
\begin{document}

\thispagestyle{empty}

\begin{center}
{\sc\large Adara M. Blaga}, \
{\sc\large Antonella Nannicini}  \
\end{center}
\vspace {1.1cm}

\centerline{\large{\textbf{On the Geometry of Metallic Pseudo-Riemannian Structures}}}

\renewcommand{\thefootnote}{\fnsymbol{footnote}}

\renewcommand{\thefootnote}{\arabic{footnote}}
\setcounter{footnote}{0}

\vspace{0.6cm}
\begin{center}
\begin{minipage}[t]{11cm}
\small{
\noindent \textbf{Abstract.}
We generalize the notion of metallic structure in the pseudo-Riemannian setting, define the metallic Norden structure
and study its integrability. We consider metallic maps between metallic manifolds and give conditions under which they are constant. We also construct a metallic natural connection recovering as particular case the
Ganchev and Mihova connection, which we extend to a metallic natural connection on the generalized tangent bundle.
Moreover, we construct metallic pseudo-Riemannian structures on the tangent and cotangent bundles.
\medskip


\medskip

\noindent \textbf{Mathematics~Subject~Classification~(2010):}
53C15, 53C25.

}
\end{minipage}
\end{center}

\bigskip

\section{Introduction}

On a smooth manifold, $M$, the concept of almost complex structure can be generalized to almost product, almost tangent and to some other general polynomial structures, as $C^{\infty}$ tensor fields $J$, of $(1,1)$-type, such that:
$$J^n+a_{n-1}J^{n-1}+...+a_1J+a_0I=0,$$
where $I$ is the identity on the Lie algebra of vector fields on $M$, $C^{\infty}(TM)$, $a_0,...,a_{n-1}$ are real numbers and $n\ge 2$.

In this paper we consider the case $n=2$:
$$J^2=pJ+qI,$$
where $p$ and $q$ are real numbers. In particular, for $p=0$ and $q=-1$, $J$ is an almost complex structure on $M$.\\

We recall that, for fixed positive integer numbers $p$ and $q$, the $(p, q)$-{\it metallic number} stands for the positive solution of the equation
$x^{2}-px-q=0$ and it is equal to
\begin{equation}
\sigma_{p,q}=\frac{p+\sqrt{p^{2}+4q}}{2} \label{20}.
\end{equation}

For particular values of $p$ and $q$, some important members of the metallic mean family \cite{Spinadel1} are the followings: \textit{the Golden mean} $\phi =\frac{1+\sqrt{5}}{2}$ for $p=q=1$, \textit{the Silver mean} $\sigma _{Ag}=\sigma _{2, 1}=1+\sqrt{2}$ for $q=1$ and $p=2$, \textit{the Bronze mean} $\sigma _{Br}=\sigma _{3, 1}=\frac{3+\sqrt{13}}{2}$ for $q=1$ and $p=3$, \textit{the Subtle mean} $\sigma_{4, 1}=2+\sqrt{5}=\phi^{3}$  for $p=4$ and $q=1$, \textit{the Copper mean} $\sigma _{Cu}=\sigma _{1, 2}=2$ for $p=1$ and $q=2$, \textit{the Nickel mean} $\sigma _{Ni}=\sigma _{1, 3}=\frac{1+\sqrt{13}}{2}$ for $p=1$ and $q=3$ and so on.

Extending this idea to tensor fields, C.-E. Hre\c tcanu and M. Crasmareanu introduced the notion of metallic structure:
\begin{definition} \cite{c}
A $(1,1)$-tensor field $J$ on $M$ is called a \textit{metallic structure} if it satisfies the equation:
\begin{equation}
J^2=pJ+qI, \label{40}
\end{equation}
for $p$ and $q$ positive integer numbers, where $I$ is the identity operator on $C^{\infty}(TM)$. In this case, the pair $(M, J)$ is called a {\it metallic manifold}. Moreover, if a Riemannian metric $g$ on $M$ is compatible with $J$, that is $g(JX, Y)=g(X, JY)$, for any $X$, $Y\in C^{\infty}(TM)$, we call the pair $(J,g)$ a {\it metallic Riemannian structure} and $(M, J,g)$ a {\it metallic Riemannian manifold}.
\end{definition}
From the compatibility condition, we immediately get that a metallic Riemannian structure satisfies
$$g(JX,JY)=pg(X,Y)+qg(X,JY),$$ for any $X$, $Y \in C^{\infty}(TM)$.

 $(1,1)$-tensor fields on $M$, which are $g$-symmetric, have applications in Generalized Geometry since they naturally define (pseudo-)calibrated generalized complex structures (\cite{n0}, \cite{n}).

In this paper we generalize the notion of metallic structure in the pseudo-Riemannian setting in order to include Norden structures. We define the concept of metallic Norden structure and study its integrability. We consider metallic maps between metallic manifolds and give conditions under which they are constant. We also construct a metallic natural connection recovering, as particular case, the
Ganchev and Mihova connection defined for Norden structures and we extend it to a metallic natural connection on the generalized tangent bundle.
Moreover, we construct metallic pseudo-Riemannian structures on the tangent and cotangent bundles.

\section{Metallic pseudo-Riemannian manifolds}

The notion of metallic Riemannian manifold can be generalized to a metallic pseudo-Riemannian manifold. We pose the following:

\begin{definition} Let $(M,g)$ be a pseudo-Riemannian manifold and let $J$ be a $g$-symmetric $(1,1)$-tensor field on $M$ such that $J^2=pJ+qI$, for some $p$ and $q$ real numbers. Then the pair $(J,g)$ is called a \textit{metallic pseudo-Riemannian structure on $M$} and $(M,J,g)$ is called a \textit{metallic pseudo-Riemannian manifold}.
\end{definition}

Fix now a metallic structure $J$ on $M$ and define the associated linear connections as follows:

\begin{definition}

i) A linear connection $\nabla$ on $M$ is called $J$-{\it connection} if $J$ is covariantly constant with respect to $\nabla$,
namely $\nabla J=0$.

ii) A  metallic pseudo-Riemannian manifold $(M,J,g)$ such that the Levi-Civita connection $\nabla$ with respect to $g$ is a $J$-connection is called a \textit{locally metallic pseudo-Riemannian manifold}.
\end{definition}

The concept of integrability is defined in the classical manner:

\begin{definition}
A metallic structure $J$ is called {\it integrable} if its Nijenhuis tensor field $N_J$ vanishes, where
$N_{J}(X, Y):=[JX, JY]-J[JX, Y]-J[X, JY] +J^{2}[X, Y]$, for $X$, $Y\in C^{\infty}(TM)$.
\end{definition}

\begin{lemma} If $(M,J,g)$ is a locally metallic pseudo-Riemannian manifold, then $J$ is integrable.
\end{lemma}
\begin{proof} We have:
$$N_J(X,Y)=({\nabla}_{JX}J)Y-({\nabla}_{JY}J)X+J({\nabla}_{Y}J)X-J({\nabla}_{X}J)Y,$$
for any $X$, $Y \in C^{\infty}(TM)$. Then the statement.
\end{proof}

\begin{remark} Every pseudo-Riemannian manifold admits locally metallic pseudo-Riemannian structures, namely $J=\mu I$, where $\mu=\frac{p\pm\sqrt{p^{2}+4q}}{2}$ with $p^2+4q \geq 0$.
\end{remark}

\begin{definition} $J:=\mu I$, where $\mu=\frac{p\pm\sqrt{p^{2}+4q}}{2}$ with $p^2+4q \geq 0$, is called a \textit{trivial metallic structure}.
\end{definition}

\begin{definition}
A  metallic pseudo-Riemannian manifold $(M,J,g)$ such that the Levi-Civita connection $\nabla$ with respect to $g$ satisfies the condition
$$({\nabla}_XJ)Y+({\nabla}_YJ)X=0,$$
for any $X$, $Y\in C^{\infty}(TM)$, is called a \textit{nearly locally metallic pseudo-Riemannian manifold}.
\end{definition}

\begin{proposition}
A nearly locally metallic pseudo-Riemannian manifold $(M,J,g)$ such that $J^2=pJ+qI$ with $p^2+4q>0$ is a locally metallic pseudo-Riemannian manifold if and only if $J$ is integrable.
\end{proposition}
\begin{proof} For any $X$, $Y \in C^{\infty}(TM)$, we have:
$$N_J(X,Y)=({\nabla}_{JX}J)Y-({\nabla}_{JY}J)X+J({\nabla}_{Y}J)X-J({\nabla}_{X}J)Y=$$
$$=-({\nabla}_{Y}J)JX+({\nabla}_{X}J)JY+J({\nabla}_{Y}J)X-J({\nabla}_{X}J)Y=$$
$$=-({\nabla}_{Y}J^2X)+2(J{\nabla}_{Y}J)X+J^2{\nabla}_{Y}X+({\nabla}_{X}J^2Y)-2(J{\nabla}_{X}J)Y-J^2{\nabla}_{X}Y=$$
$$=-2p({\nabla}_{Y}J)X+4J({\nabla}_{Y}J)X=2(2J-pI)({\nabla}_{Y}J)X.$$
We observe that $p\over 2$ is not an eigenvalue of $J$ because $p^2+4q > 0$, thus we get that if $J$ is nearly locally metallic, then
$$N_J=0  \Longleftrightarrow \nabla J=0$$
and the proof is complete.
\end{proof}

\section{Metallic natural connection}

\begin{theorem} Let $(M,J,g)$  be a metallic pseudo-Riemannian manifold such that $J^2=pJ+qI$ with $p^2+4q \neq 0$. Let $\nabla$ be the Levi-Civita connection of $g$ and let $D$ be the linear connection defined by:
\begin{equation}
D:={\nabla}+\frac{2}{p^2+4q} J(\nabla J)- \frac{p}{p^2+4q} (\nabla J). \label{500}
\end{equation}
Then
\begin{equation}\left \{ \begin{array}{l}DJ=0\\

\vspace{0.1cm}
Dg=0.

\end{array}
\right.
\end{equation}
\end{theorem}
\begin{proof} We have:
$$DJ-JD=\nabla J- J \nabla + \frac{1}{p^2+4q}(2J\nabla J^2-2q\nabla J-p \nabla J^2-2J^2 \nabla J+pJ^2 \nabla +2qJ\nabla)=$$
$$=\nabla J- J \nabla + \frac{1}{p^2+4q}(2qJ\nabla -2q\nabla J-p^2 \nabla J-2q\nabla J+p^2 J \nabla +2qJ\nabla)=$$
$$=\nabla J- J \nabla + \frac{1}{p^2+4q}[(4q+p^2)J\nabla -(4q+p^2)\nabla J]=0.$$
Moreover, for any $X$, $Y$, $Z \in C^{\infty}(TM)$, we have:
$$(D_Xg)(Y,Z)=X(g(Y,Z))-g(D_X Y,Z)-g(Y,D_X Z)=$$
$$= \frac{1}{p^2+4q}[2g(({\nabla}_X J)Y,JZ)-pg(({\nabla}_X J)Y,Z)+2g(JY,({\nabla}_X J)Z)-pg(Y,({\nabla}_X J)Z)]=$$
$$=\frac{1}{p^2+4q}[2g({\nabla}_X JY,JZ)-pg({\nabla}_X Y,JZ)-2qg({\nabla}_X Y,Z)-pg({\nabla}_X JY,Z)+$$
$$+2g(JY,{\nabla}_X JZ)-pg(JY,{\nabla}_X Z)-2qg(Y,{\nabla}_X Z)-pg(Y,{\nabla}_XJ Z)]=$$
$$=\frac{1}{p^2+4q}[2X(g(JY,JZ))-pX(g(Y,JZ))-pX(g(JY,Z))-2qX(g(Y,Z))]=$$
$$=\frac{1}{p^2+4q}[2X(pg(Y,JZ)+qg(Y,Z))-pX(g(Y,JZ))-qX(g(Y,Z))]=0.$$
Then the proof is complete.
\end{proof}

\begin{definition} The linear connection $D$ defined by (\ref{500}) is called the \textit{metallic natural connection} of $(M,J,g)$.
\end{definition}

A direct computation gives the following expression for the torsion $T^D$ of the natural connection $D$:
$$T^D(X,Y)=\frac{1}{p^2+4q}\{(2J-pI)({\nabla}_X JY-{\nabla}_Y JX)-(pJ+2qI)[X,Y]\},$$
for any $X$, $Y \in C^{\infty}(TM).$

Thus we get the following:
\begin{proposition} Let $(M,J,g)$ be a metallic pseudo-Riemannian manifold such that $J^2=pJ+qI$ with $p^2+4q \neq 0$. Then the torsion $T^D$ of the natural connection $D$ satisfies the following relation:
$$T^D(JX,Y)+T^D(X,JY)-pT^D(X,Y)=(2J-pI)N_J(X,Y),$$
for any $X$, $Y \in C^{\infty}(TM).$
In particular, if $J$ is integrable, then:
$$T^D(JX,Y)+T^D(X,JY)=pT^D(X,Y).$$
\end{proposition}

\begin{remark} If $p=0$, $q=-1$ and $J$ is integrable, then the natural connection $D$ coincides with the natural canonical connection defined by Ganchev and Mihova in \cite{GM}.
\end{remark}

\section{Metallic Norden structures}

Recall that a Norden manifold $(M,J,g)$ is an almost complex manifold $(M,J)$ with a neutral pseudo-Riemannian metric $g$ such that $g(JX,Y)=g(X,JY)$, for any $X$, $Y \in C^{\infty}(TM)$. We can state:

\begin{proposition}
If $(M,J,g)$ is a Norden manifold, then for any real numbers $a$ and $b$,
$$J_{a,b}:=aJ+bI$$
are metallic pseudo-Riemannian structures on $M$.
\end{proposition}
\begin{proof}
We have:
$$J_{a,b}^2=2bJ_{a,b}-(a^2+b^2)I.$$
Moreover, for any $X$, $Y \in C^{\infty} (TM)$, we have:
$$g(J_{a,b}X,Y)=g(X,J_{a,b}Y).$$
Then the statement.
\end{proof}

We remark that $J=J_{1,0}$.

Also, from $\nabla J_{a,b}=a\nabla J$ and $N_{J_{a,b}}=a^2N_J$, we get the following:

\begin{proposition}
Assume that $a\neq 0$. Then:
\begin{enumerate}
  \item $J_{a,b}$ is integrable if and only if $J$ is integrable.
  \item $J_{a,b}$ is locally metallic if and only if $J$ is K\"{a}hler.
  \item $J_{a,b}$ is nearly locally metallic if and only if $J$ is nearly K\"{a}hler.
\end{enumerate}
\end{proposition}

Conversely, we have:
\begin{proposition} If $(M,J,g)$ is a metallic pseudo-Riemannian manifold such that $J^2=pJ+qI$ with $p^2+4q <0$, then
$$J_{\pm}:=\pm({2\over{\sqrt{-p^2-4q}}}J-{p\over{\sqrt{-p^2-4q}}}I)$$
are Norden structures on $M$ and $J=aJ_{\pm}+bI$ with $a=\pm({2\over{\sqrt{-p^2-4q}}})^{-1}$ and $b=-{p\over 2}$.
\end{proposition}
\begin{proof} We have:
$${J^2_{\pm}}={1\over{-p^2-4q}}(4{J}^2-4p J +p^2I)={1\over{-p^2-4q}}(4qI +p^2I)=-I.$$
Moreover, for any $X$, $Y \in C^{\infty} (TM)$, we have:
$$g(J_\pm X,Y)=g(X,J_\pm Y).$$
Finally, we have:
$$J=\pm ({2\over{\sqrt{-p^2-4q}}})^{-1} J_{\pm}-{p \over 2} I.$$
Then the statement.
\end{proof}

We give the following definition:

\begin{definition} Let $(M,J,g)$ be a metallic pseudo-Riemannian manifold such that ${J}^2=p{J}+qI$ with $p^2+4q < 0$. Then $J$ is called a {\textit{metallic Norden structure on $M$}} and $(M,J,g)$ is called a
\textit{metallic Norden manifold}.
\end{definition}

\section{Metallic maps}

\begin{definition}
A smooth map $f:(M_1,J_1)\rightarrow (M_2,J_2)$ between two metallic manifolds is called a \textit{metallic map} if:
\begin{equation}
f_{*}\circ J_1=J_2\circ f_{*}.
\end{equation}
\end{definition}

\begin{remark}
If $f:(M_1,J_1)\rightarrow (M_2,J_2)$ is a metallic map and $J_i^2=p_iJ_i+q_iI$ with $p_i$ and $q_i$ real numbers, $i=1,2$, then:

i) $f_{*}\circ J_1^{2k+1}=J_2^{2k+1}\circ f_{*}$, for any $k\in \mathbb{N}$;

ii) $([(p_2^2+q_2)-(p_1^2+q_1)]J_1+(p_2q_2-p_1q_1)I)(TM_1)\subset \ker f_*$;

iii) in the particular case when one the structure is product and the other one is complex, then $Im J_1\subset \ker f_*$.
\end{remark}

Then we can state the following:

\begin{proposition} Let $f:(M_1,J_1)\rightarrow (M_2,J_2)$ be a metallic map between two connected metallic manifolds. If $p_1 \neq p_2$ and ${\frac{q_1-q_2}{p_2-p_1}}\neq{\frac{p_2\pm{\sqrt{{p_2}^2+4q_2}}}{2}}$, then $f$ is constant. Moreover, if $p_1=p_2$ but $q_1\neq q_2$, then $f$ is constant.
\end{proposition}
\begin{proof} From $f_{*}\circ J_1=J_2\circ f_{*}$, we get:
$$J_2\circ f_{*}\circ J_1=J_2^2\circ f_{*}$$
thus
$$f_{*}\circ J_1^2=J_2^2\circ f_{*}$$
or
$$p_1f_{*}\circ J_1+q_1f_{*}=p_2J_2\circ f_{*}+q_2f_{*}$$
equivalent to
$$(p_1-p_2)J_2\circ f_{*}+(q_1-q_2)f_{*}=0$$
and to
$$((p_1-p_2)J_2+(q_1-q_2)I)\circ f_{*}=0.$$

In particular, if $((p_1-p_2)J_2+(q_1-q_2)I)$ is invertible, then $f_{*}=0$.

Now, $(p_1-p_2)J_2+(q_1-q_2)I$ is invertible if and only if ${\frac{q_1-q_2}{p_2-p_1}}$ is not an eigenvalue of $J_2$, or $p_1=p_2$ but $q_1\neq q_2$. Then the statement.
\end{proof}

\section{Induced structures on $TM\oplus T^*M$}

\subsection{Generalized metallic pseudo-Riemannian structures}

In \cite{bn} we introduced the notion of generalized metallic structure and generalized metallic Riemannian structure. We pose the following:
\begin{definition} A pair $(\tilde J, \tilde g)$ of a generalized metallic structure $\tilde J$ and a pseudo-Riemannian metric $\tilde g$
such that $\tilde J$ is $\tilde g$-symmetric is called a \textit{generalized metallic pseudo-Riemannian structure}. If ${\tilde J}^2=p{\tilde J}+qI$ with $p^2+4q < 0$, then $\tilde J$ is called a \textit{generalized metallic Norden structure}.
\end{definition}

Let $(M,J,g)$ be a Norden manifold and let $(\tilde J, \tilde g)$ be the generalized Norden structure defined in \cite{n}:
$$\tilde J:= \begin{pmatrix} J & 0 \cr
\flat_g & -J^{*} \cr
\end{pmatrix}$$
\begin{equation}\tilde g(X+\alpha, Y+\beta):=g(X,Y)+{1\over 2}g(JX,{\sharp}_g \beta)+{1\over 2}g({\sharp}_g \alpha, JY)+g({\sharp}_g \alpha, {\sharp}_g \beta), \label{550}
\end{equation}
for any $X$, $Y\in C^{\infty} (TM)$ and $\alpha$, $\beta \in C^{\infty}(T^*M)$. Then
$\tilde J$ defines the following family of generalized metallic Norden structures:
$${\tilde J}_{a,b}:=a\tilde J+bI= \begin{pmatrix}aJ+bI & 0 \cr
a\flat_g & -aJ^{*}+bI \cr
\end{pmatrix},$$
where $a$ and $b$ are real numbers, since
$${{\tilde J}_{a,b}}^2=p{{\tilde J}_{a,b}}+qI$$
with $p=2b$ and $q=-(a^2+b^2)$ and
$$\tilde g({\tilde J}_{a,b}(\sigma),\tau)=\tilde g(\sigma,{\tilde J}_{a,b}(\tau)),$$
for any $\sigma$, $\tau \in C^{\infty} (TM \oplus T^*M)$.

We remark that, up to rescaling the metric, instead of ${\tilde J}_{a,b}$, we can consider the family:
$${\hat J}_{a,b}:= \begin{pmatrix}aJ+bI & 0 \cr
\flat_g & -aJ^{*}+bI \cr
\end{pmatrix}.$$

Moreover, if $(M,J,g)$ is a metallic pseudo-Riemannian manifold with $J^2=pJ+qI$, for $p=2b$ and $q=-(a^2+b^2)$, we immediately have that
\begin{equation}\label{c}
{\hat J}:= \begin{pmatrix}J & 0 \cr
\flat_g & -J^{*}+pI \cr
\end{pmatrix}
\end{equation}
is a generalized metallic structure with:
$${\hat J}^2=p\hat J+qI.$$

If we assume $p^2+4q\neq 0$, generalizing (\ref{550}), we define the pseudo-Riemannian metric:
\begin{equation}\hat g(X+\alpha,Y+\beta):=
\label{580}\end{equation}
$$=g(X,Y)+g({\sharp}_g \alpha, {\sharp}_g \beta)+{\frac{p}{p^2+4q}}[g(X,{\sharp}_g \beta)+g(Y,{\sharp}_g \alpha)]+$$
$$-{\frac{2}{p^2+4q}}[g(JX,{\sharp}_g \beta)+g(JY,{\sharp}_g \alpha)]=$$
$$=g(X,Y)+g({\sharp}_g \alpha, {\sharp}_g \beta)+{\frac{p}{p^2+4q}}(\alpha (Y)+\beta(X))+$$
$$-{\frac{2}{p^2+4q}}(\alpha (JY)+\beta (JX)),$$
for any $X$, $Y\in C^{\infty} (TM)$ and $\alpha$, $\beta \in C^{\infty}(T^*M)$ and we have the following:
\begin{proposition}
Let $(M,J,g)$ be a metallic pseudo-Riemannian manifold such that $J^2=pJ+qI$ with $p^2+4q\neq 0$. Then
$(\hat{J},\hat{g})$ is a generalized metallic pseudo-Riemannian structure with $\hat{J}$ given by (\ref{c}) and $\hat{g}$ given by (\ref{580}).
\end{proposition}
\begin{proof}
For any $X+\alpha$, $Y+\beta\in C^{\infty}(TM\oplus T^*M)$, we have:
$$\hat g(\hat J (X+\alpha), Y+\beta)=\hat{g}(JX+\flat_g(X)-J^*(\alpha)+p\alpha,Y+\beta)=$$
$$=g(JX,Y)+\beta (X)-g({\sharp}_gJ^*(\alpha),{\sharp}_g \beta)+pg({\sharp}_g \alpha,{\sharp}_g \beta)+$$
$$+{p\over {p^2+4q}}[\beta(JX)+g(X,Y)-\alpha (JY)+p\alpha (Y)]-{2\over {p^2+4q}}[p\beta (JX)+q\beta (X)+$$
$$+g(X,JY)-q \alpha (Y)]$$
and
$$\hat g(X+\alpha,\hat J (Y+\beta))=g(X+\alpha,JY+\flat_g(Y)-J^*(\beta)+p\beta)=$$
$$=g(X,JY)+\alpha (Y)-g({\sharp}_gJ^*(\beta),{\sharp}_g \alpha)+pg({\sharp}_g \alpha,{\sharp}_g \beta)+$$
$$+{p\over {p^2+4q}}[\alpha (JY)+g(X,Y)-\beta(JX)+p\beta (X)]-{2\over {p^2+4q}}[p\alpha (JY)+q\alpha (Y)+$$
$$+g(Y,JX)-q \beta (X)].$$
Since $(M,J,g)$ is a metallic pseudo-Riemannian manifold and $J^*={\flat}_gJ {\sharp}_g$ we get the statement.
\end{proof}

\begin{definition} Let $(M,J,g)$ be a metallic pseudo-Riemannian manifold. The pair $(\hat J, \hat g)$ defined by (\ref{c}) and (\ref{580}) is called the \textit{generalized metallic pseudo-Riemannian structure defined by} $(J,g)$.
\end{definition}

More generally, we can construct generalized metallic pseudo-Riemannian structures by using a pseudo-Riemannian metric $g$ on $M$ and an arbitrary $g$-symmetric endomorphism $J$ of the tangent bundle as in the followings.

\begin{theorem}
Let $g$ be a pseudo-Riemannian metric on $M$ and let $J$ be an arbitrary endomorphism of the tangent bundle which is $g$-symmetric. Then
$(\check{J},\check{g})$ is a generalized metallic pseudo-Riemannian structure, where
\begin{equation}
\check{J}:=\begin{pmatrix} J & (-J^2+pJ+qI)\sharp_g \cr
\flat_g & -J^*+pI \cr
\end{pmatrix}
\end{equation}
and
\begin{equation}
\check{g}(X+\alpha,Y+\beta):=\label{560}
\end{equation}
$$=g(X,Y)+\frac{p^2+4q}{4}g(\sharp_g\alpha,\sharp_g\beta)+\frac{p}{4}(\alpha(Y)+\beta(X))-\frac{1}{2}(\alpha(JY)+\beta(JX)),$$
for any $X$, $Y\in C^{\infty} (TM)$ and $\alpha$, $\beta \in C^{\infty}(T^*M)$ and $p$ and $q$ any fixed real numbers.

Moreover, $\check{J}$ satisfies:
$$(\check{J}(X+\alpha),Y+\beta)+(X+\alpha,\check{J}(Y+\beta))=p\cdot(X+\alpha,Y+\beta),$$
where
\begin{equation}\label{g3}
(X+\alpha,Y+\beta):=-\frac{1}{2}(\alpha(Y)-\beta(X))
\end{equation}
is the natural symplectic structure on $TM\oplus T^*M$.
\end{theorem}
\begin{proof}
A direct computation gives ${\check J}^2=p \check J+q I$.

Moreover, for any $X+\alpha\in C^{\infty}(TM\oplus T^*M)$, we have:
$$\check{J}(X+\alpha)=JX-J^2(\sharp_g\alpha)+pJ(\sharp_g\alpha)+q\sharp_g\alpha+\flat_g(X)-J^*(\alpha)+p\alpha$$
and using the definition of $(\cdot,\cdot)$ we get the last statement.

Now, for any $X+\alpha$, $Y+\beta\in C^{\infty}(TM\oplus T^*M)$, we have:
$$\check{g}(\check{J}(X+\alpha),Y+\beta)=g(JX,Y)-g(J^2(\sharp_g \alpha),Y)+
pg(J(\sharp_g \alpha),Y)+$$$$+qg(\sharp_g\alpha,Y)+\frac{p^2+4q}{4}g(X,\sharp_g\beta)-\frac{p^2+4q}{4}g(\sharp_gJ^*(\alpha),\sharp_g\beta)+\frac{p^2+4q}{4}pg(\sharp_g\alpha,\sharp_g\beta)
+$$$$+\frac{p}{4}[g(X,Y)-\alpha(JY)+p\alpha(Y)+\beta(JX)-\beta(J^2(\sharp_g\alpha))+p\beta(J(\sharp_g\alpha))+q\beta(\sharp_g\alpha)]
-$$$$-\frac{1}{2}[g(X,JY)-\alpha(J^2Y)+p\alpha(JY)+\beta(J^2X)-\beta(J^3(\sharp_g\alpha))+p\beta(J^2(\sharp_g\alpha))+q\beta(J(\sharp_g\alpha))]$$
and
$$\check{g}(X+\alpha,\check{J}(Y+\beta))=g(JY,X)-g(J^2(\sharp_g \beta),X)+
pg(J(\sharp_g \beta),X)+$$$$+qg(\sharp_g\beta,X)+\frac{p^2+4q}{4}g(Y,\sharp_g\alpha)-\frac{p^2+4q}{4}g(\sharp_gJ^*(\beta),\sharp_g\alpha)+\frac{p^2+4q}{4}pg(\sharp_g\beta,\sharp_g\alpha)
+$$$$+\frac{p}{4}[g(Y,X)-\beta(JX)+p\beta(X)+\alpha(JY)-\alpha(J^2(\sharp_g\beta))+p\alpha(J(\sharp_g\beta))+q\alpha(\sharp_g\beta)]
-$$$$-\frac{1}{2}[g(Y,JX)-\beta(J^2X)+p\beta(JX)+\alpha(J^2Y)-\alpha(J^3(\sharp_g\beta))+p\alpha(J^2(\sharp_g\beta))+q\alpha(J(\sharp_g\beta))].$$
Since $J$ is $g$-symmetric and $J^*=\flat_gJ\sharp_g$ we get the statement.
\end{proof}

\begin{remark}
i) For $p=0$, the structure $\check{J}$ is anti-calibrated
with respect to (\ref{g3}).

ii) In particular, for $p=0$ and $q=1$ we get the generalized product structure
$$\check{J}_p:=\begin{pmatrix} J & (-J^2+I) \sharp_g \cr
\flat_g & -J^* \cr
\end{pmatrix}$$
and for $p=0$ and $q=-1$ we get the generalized complex structure
$$\check{J}_c:=\begin{pmatrix} J & (-J^2-I)\sharp_g \cr
\flat_g & -J^* \cr
\end{pmatrix}$$
which are both anti-calibrated.

iii) If $J$ is a metallic structure with $J^2=pJ+qI$, then $\check{J}=\hat{J}$.
\end{remark}

\begin{remark}
i) Notice that if $(J,g)$ is a metallic pseudo-Riemannian structure such that $J^2=pJ+qI$ with $p^2+4q<0$, then $\hat{J}'$ obtained from $\hat{J}_{a,b}$ (beside $\hat J$) and given by:
\begin{equation}\label{ecu}
\hat{J}':=\begin{pmatrix} -J+pI & 0 \cr
\flat_g & J^* \cr
\end{pmatrix}
\end{equation}
is a generalized metallic structure with:
$${\hat J}'^2=p\hat J'+qI.$$

Moreover, the two structures $\hat{J}$ and $\hat{J}'$ coincide with the ones obtained by considering first the Norden structures $J_{\pm}$ induced by $J$ and then defining the generalized metallic structures
\begin{equation}
\hat{J}_{a,b}:=\begin{pmatrix} aJ_{\pm}+bI & 0 \cr
\flat_g & -aJ_{\pm}^*+bI \cr
\end{pmatrix},
\end{equation}
where $a=\pm \frac{\sqrt{-p^2-4q}}{2}$ and $b=\frac{p}{2}$.

ii)
The structure $\check{J}'$ defined by:
\begin{equation}
\check{J}':=\begin{pmatrix} -J+pI & (-J^2+pJ+qI)\sharp_g \cr
\flat_g & J^* \cr
\end{pmatrix}
\end{equation}
is also a generalized metallic structure and for the particular case when $J$ is metallic with $J^2=pJ+qI$, it is precisely $\hat{J}'$.
\end{remark}

\begin{remark} Let $(M,g)$ be a pseudo-Riemannian manifold and let $J$ be an arbitrary $g$-symmetric endomorphism of the tangent bundle. Then for any $p$ and $q$ real numbers with $p^2+4q<0$:
$$\check J_{\pm}:=\pm({2\over{\sqrt{-p^2-4q}}}\check J-{p\over{\sqrt{-p^2-4q}}}I)$$
are generalized Norden structures with respect to the metric $\check g$.
\end{remark}

\subsection{Generalized metallic natural connection}

Let $(M,J,g)$ be a metallic pseudo-Riemannian manifold and let $D$ be the metallic natural connection given by (\ref {500}). We define:
$$\hat D : C^{\infty}(TM\oplus T^*M) \times C^{\infty}(TM\oplus T^*M)\rightarrow C^{\infty}(TM\oplus T^*M) \label{133}
$$
by:
\begin{equation}
{\hat D}_{X+\alpha} (Y+\beta):={D_X} Y+{D_X} {\beta},  \label{132}
\end{equation}
for any $X$, $Y\in C^{\infty} (TM)$ and $\alpha$, $\beta \in C^{\infty}(T^*M)$ and we have:

\begin{theorem}
The linear connection $\hat D$ satisfies the following conditions:
\begin{equation}\left \{ \begin{array}{l}{\hat D}{\hat J}=0\\

\vspace{0.1cm}
{\hat D}{\hat g}=0.

\end{array}
\right.\label{135}
\end{equation}
Moreover, $T^{\hat D}(X+\alpha,Y+\beta)=T^D(X,Y)$, for any $X$, $Y\in C^{\infty} (TM)$ and $\alpha$, $\beta \in C^{\infty}(T^*M)$, and $\hat D$ is flat if and only if $D$ is flat.
\end{theorem}
\begin{proof}
From the definition of $\hat J$ and from the properties of $D$ we get:
$$(\hat D_{X+\alpha} {\hat J})(Y+ \beta)={\hat D}_{X+\alpha} (JY+\flat_g(Y)-J^*(\beta)+p\beta)=D_X JY+D_X(\flat_g(Y)-J^*(\beta)+p\beta)=$$
$$=J(D_X Y)+\flat_g(D_X Y)-J^*(D_X \beta)+pD_X \beta=\hat J(D_X Y)+\hat J(D_X \beta)=\hat J({\hat D}_{X+\alpha} (Y+\beta)),$$
for any $X$, $Y\in C^{\infty} (TM)$ and $\alpha$, $\beta \in C^{\infty}(T^*M)$.

Moreover, from the definition of $\hat g$ we get:
$$X(\hat g(Y+\beta,Z+\gamma))-\hat g({\hat D}_{X+\alpha}(Y+\beta),Z+\gamma)-\hat g(Y+\beta, {\hat D}_{X+\alpha} (Z+\gamma))=$$
$$=X(g(Y,Z)+g({\sharp}_g \beta,{\sharp}_g \gamma)+{p\over {p^2+4q}}(\beta (Z)+\gamma(Y))-{2\over {p^2+4q}}(\beta (JZ)+\gamma (JY))-$$
$$-[g({D}_X Y,Z)+g({\sharp}_g (D_X\beta),{\sharp}_g \gamma)+{p\over {p^2+4q}}((D_X \beta)Z+\gamma(D_X Y))-$$
$$-{2\over {p^2+4q}}((D_X \beta) (JZ)+\gamma (J(D_X Y)))]-$$
$$-[g(Y,{D}_XZ)+g({\sharp}_g\beta,{\sharp}_g (D_X \gamma))+{p\over {p^2+4q}}((D_X \gamma) Y+\beta(D_X Z))-$$
$$-{2\over {p^2+4q}}((D_X \gamma) (JY)+\beta (J(D_X Z)))]=0,$$
for any $X$, $Y$, $Z\in C^{\infty} (TM)$ and $\alpha$, $\beta$, $\gamma \in C^{\infty}(T^*M)$.

If we define by:
$$T^{\hat D}(X+\alpha,Y+\beta):=\hat{D}_{X+\alpha}(Y+\beta)-\hat{D}_{Y+\beta}(X+\alpha)-[X+\alpha,Y+\beta]_{D}$$
the torsion of $\hat{D}$, where
$$[X+\alpha,Y+\beta]_{D}:=[X,Y]+D_X \beta-D_Y\alpha,$$
we have:
$$T^{\hat D}(X+\alpha,Y+\beta)=T^D(X,Y).$$

Also, if we define by:
$$R^{\hat{D}}(X+\alpha,Y+\beta)(Z+\gamma):=$$
$$=\hat{D}_{X+\alpha}\hat{D}_{Y+\beta}(Z+\gamma)-
\hat{D}_{Y+\beta}\hat{D}_{X+\alpha}(Z+\gamma)-\hat{D}_{[X+\alpha,Y+\beta]_D}(Z+\gamma)$$
the curvature of $\hat{D}$, we have:
$$R^{\hat{D}}(X+\alpha,Y+\beta)(Z+\gamma)=R^{D}(X,Y)Z+R^{D}(X,Y)\gamma$$
and we get the statement.
\end{proof}

\begin{definition}
The linear connection $\hat D$ defined by (\ref{132}) is called the \textit{generalized metallic natural connection} of $(TM\oplus T^*M, \hat J, \hat g)$.
\end{definition}

\section{Metallic pseudo-Riemannian structures on tangent and cotangent bundles}

\subsection{Metallic pseudo-Riemannian structures on the tangent bundle}

Let $(M,g)$ be a pseudo-Riemannian manifold and let $\nabla$ be a linear connection on $M$. Then $\nabla$ defines the decomposition into the horizontal and vertical subbundles of $T(TM)$:
$$ T(TM)=T^H(TM)\oplus T^V(TM).$$

Let $\pi :TM \rightarrow M$ be the canonical projection and ${\pi}_*:T(TM) \rightarrow TM$ be the tangent map of $\pi$. If $a \in TM$ and $A \in T_a(TM)$, then ${\pi}_*(A) \in T_{\pi (a)}M$ and we denote by ${\chi}_a$ the standard identification between $T_{\pi (a)}M$ and its tangent space
$ T_a(T_{\pi (a)}M)$.

Let ${\Psi}^{\nabla} :TM \oplus T^{*}M \rightarrow T(TM)$ be the bundle morphism defined by:
$${\Psi}^{\nabla}(X+\alpha):=X^H_a+{\chi}_a({\sharp}_g \alpha),$$
where $a \in TM$ and $X^H_a$ is the horizontal lifting of $X\in T_{\pi (a)}M$.

Let $\left\{ x^{1},...,x^{n}\right\} $ be local coordinates on $M$, let $\left\{ {\tilde{x}}^{1},..., {\tilde{x}}^{n},y^1,...,y^n\right\} $ be respectively the corresponding local coordinates on $TM$ and let\\
$\{X_1,...,X_n, \dfrac{\partial }{\partial
{y^{1}}},.., \dfrac{\partial }{\partial
{y^{n}}}\}$ be a local frame on $T(TM)$, where $X_i=\dfrac{\partial }{\partial
{\tilde{x}}^{i}}$.

We have:
$${\Psi}^{\nabla}\left(\dfrac{\partial }{\partial
{x^{i}}}\right)=X_i^H$$
$${\Psi}^{\nabla}\left({dx^j}\right)=g^{jk} \dfrac{\partial}{\partial y^k}.$$

Let $J$ be an arbitrary $g$-symmetric endomorphism on the tangent bundle. For any $p$ and $q$ real numbers, let
$$ \check {J}:=\begin{pmatrix} J & (-J^2+pJ+qI)\sharp_g \cr
\flat_g & -J^*+pI \cr
\end{pmatrix}$$
be the generalized metallic pseudo-Riemannian structure defined by $(J,g)$ with the pseudo-Riemannian metric $\check g$ defined by (\ref {560}).
The isomorphism ${\Psi}^{\nabla}$ allows us to construct a natural metallic structure $\bar J$ and a natural pseudo-Riemannian metric $\bar g$ on $TM$ in the following way.

We define $\bar J:T(TM) \rightarrow T(TM)$ by
$${\bar J}:= ({\Psi}^{\nabla} )\circ{\hat J} \circ ({\Psi}^{\nabla} )^{-1}$$
and the pseudo-Riemannian metric $\bar g$ on $TM$ by
$${\bar g}:= (({\Psi}^{\nabla} )^{-1})^{*}(\hat g). $$

\begin{proposition} $(TM,{\bar J},\bar g)$ is a metallic pseudo-Riemannian manifold.
\end{proposition}

\begin{proof}
From the definition it follows that ${\bar J}^2=p{\bar J}+qI$ and ${\bar g}({\bar J}X,Y)={\bar g}(X,{\bar J}Y)$, for any $X,Y \in C^{\infty}(T(TM))$.
\end{proof}

In local coordinates, we have the following expressions for ${\bar J}$ and $\bar g$:
$$
\left \{
\begin{array}{l}
\vspace{0.2cm}
{\bar J} \left(  X_i^H\right)=J^k_iX_k^H+{\dfrac{\partial }{\partial
y^{i}}}\\
{\bar J} \left({\dfrac{\partial }{\partial
y^j}}\right)=(-J^2+pJ+qI)^k_jX^H_k-J^k_j{\dfrac{\partial }{\partial
y^{k}}}+p{\dfrac{\partial }{\partial
y^{j}}}
\end{array}\label{190}
\right.
$$
and
$$
\left \{
\begin{array}{l}
\vspace{0.2cm}
{\bar g} \left(  X_i^H,X^H_j\right)=g_{ij}\\
\vspace{0.2cm}
{\bar g} \left(  X_i^H,{\dfrac{\partial }{\partial
y^{j}}}\right )={\dfrac{p }{p^2+4q}}g_{ij}-{\dfrac{2 }{p^2+4q}}g_{jl}J^l_i\\
{\bar g} \left({\dfrac{\partial }{\partial
y^{i}}},{\dfrac{\partial }{\partial
y^{j}}}\right)=g_{ij}.
\end{array}\label{200}
\right.
$$

Computing the Nijenhuis tensor of ${\bar J}$, in the case when $J$ is metallic with $J^2=pJ+qI$ and $\nabla$ is the Levi-Civita connection of $g$, we get:
$$
\begin{array}{l}
\vspace{0.2cm}
N_{{\bar J}}\left({\dfrac{\partial }{\partial y^{i}}},{\dfrac{\partial }{\partial y^{j}}}\right)=0 \\
\end{array}\label{210}
$$
$$
\begin{array}{l}
\vspace{0.2cm}
N_{{\bar J}} \left( X^H_{i},{\dfrac{\partial }{\partial y^{j}}}\right)={\left( \left( {\nabla}_{JX_{i}} J \right) X_{j}-J \left( {\nabla}_{X_{i}} J \right){ X_{j}}\right)}^k {\dfrac{\partial }{\partial y^{k}}}\\
\end{array} \label{220}
$$
$$
\begin{array}{l}
\vspace{0.2cm}
N_{{\bar J}}\left(  X_i^H, X_j^H \right)=(N_J\left( {X_{i}},X_j\right))^kX_k^H -\\
- y^s \left(J^k_iJ^h_jR^r_{khs} -J_l^rJ^k_iR^l_{kjs}-J^h_jJ^r_lR^l_{ihs}+pJ^r_lR^l_{ijs}+qR^r_{ijs}\right) {{\dfrac{\partial }{\partial y^{r}}}}. \\
\end{array}\label{230}
$$

Therefore we can state the following:

\begin{proposition} Let $(M,J,g)$ be a flat locally metallic pseudo-Riemannian manifold. If $\nabla$ is the Levi-Civita connection of $g$, then $({\bar J},\bar g)$ is an integrable metallic pseudo-Riemannian structure on $TM$.
\end{proposition}

\subsection{Metallic pseudo-Riemannian structures on the cotangent bundle}

Let $(M,g)$ be a pseudo-Riemannian manifold and let $\nabla$ be a linear connection on $M$. Then $\nabla$ defines the decomposition into the horizontal and vertical subbundles of $T(T^*M)$:
$$ T(T^*M)=T^H(T^*M)\oplus T^V(T^*M).$$

Let $\pi :T^*M \rightarrow M$ be the canonical projection and ${\pi}_*:T(T^*M) \rightarrow TM$ be the tangent map of $\pi$. If $a \in T^*M$ and $A \in T_a(T^*M)$, then ${\pi}_*(A) \in T_{\pi (a)}M$ and we denote by ${\chi}_a$ the standard identification between $T^*_{\pi (a)}M$ and its tangent space $ T_a(T^*_{\pi (a)}M)$.

Let ${\Phi}^{\nabla} :TM \oplus T^{*}M \rightarrow T(T^*M)$ be the bundle morphism defined by:
$${\Phi}^{\nabla}(X+\alpha):=X^H_a+{\chi}_a(\alpha),$$
where $a \in T^*M$ and $X^H_a$ is the horizontal lifting of $X\in T_{\pi (a)}M$.

Let $\left\{ x^{1},...,x^{n}\right\} $ be local coordinates on $M$, let $\left\{ {\tilde{x}}^{1},..., {\tilde{x}}^{n},y_1,...,y_n\right\} $ be respectively the corresponding local coordinates on $T^*M$ and let\\
$\{X_1,...,X_n, \dfrac{\partial }{\partial
{y_{1}}},.., \dfrac{\partial }{\partial
{y_{n}}}\}$ be a local frame on $T(T^*M)$, where $X_i=\dfrac{\partial }{\partial
{\tilde{x}}^{i}}$.

We have:
$${\Phi}^{\nabla}\left(\dfrac{\partial }{\partial
{x^{i}}}\right)=X_i^H$$
$${\Phi}^{\nabla}\left({dx^j}\right)=\dfrac{\partial }{\partial
{y_{j}}}.$$

Let $J$ be an arbitrary $g$-symmetric endomorphism on the tangent bundle. For any $p$ and $q$ real numbers, let
$$ \check {J}:=\begin{pmatrix} J & (-J^2+pJ+qI)\sharp_g \cr
\flat_g & -J^*+pI \cr
\end{pmatrix}$$
be the generalized metallic pseudo-Riemannian structure defined by $(J,g)$ with the pseudo-Riemannian metric $\check g$ defined by (\ref {560}).
The isomorphism ${\Phi}^{\nabla}$ allows us to construct a natural metallic structure $\tilde J$ and a natural pseudo-Riemannian metric $\tilde g$ on $T^*M$ in the following way.

We define $\tilde J:T(T^*M) \rightarrow T(T^*M)$ by
$${\tilde J}:= ({\Phi}^{\nabla} )\circ{\hat J} \circ ({\Phi}^{\nabla} )^{-1}$$
and the pseudo-Riemannian metric $\tilde g$ on $T^*M$ by
$${\tilde g}:= (({\Phi}^{\nabla} )^{-1})^{*}(\hat g). $$

\begin{proposition}
$(T^*M,{\tilde J},\tilde g)$ is a metallic pseudo-Riemannian manifold.
\end{proposition}

\begin{proof}
From the definition it follows that ${\tilde J}^2=p{\tilde J}+qI$ and ${\tilde g}({\tilde J}X,Y)={\tilde g}(X,{\tilde J}Y)$, for any $X,Y \in C^{\infty}(T(T^*M))$.
\end{proof}

In local coordinates, we have the following expressions for ${\tilde J}$ and $\tilde g$:
$$
\left \{
\begin{array}{l}
\vspace{0.2cm}
{\tilde J} \left(  X_i^H\right)=J^k_iX_k^H+g_{ik}{\dfrac{\partial }{\partial
y_{k}}}\\
{\tilde J} \left({\dfrac{\partial }{\partial
y_{j}}}\right)=(-J^2+pJ+qI)^l_kg^{jk}{\dfrac{\partial }{\partial
y_{l}}}-J^j_k{\dfrac{\partial }{\partial
y_{k}}}+p{\dfrac{\partial }{\partial
y_{j}}}
\end{array}\label{190}
\right.
$$
and
$$
\left \{
\begin{array}{l}
\vspace{0.2cm}
{\tilde g} \left(  X_i^H,X^H_j\right)=g_{ij}\\
\vspace{0.2cm}
{\tilde g} \left(  X_i^H,{\dfrac{\partial }{\partial
y_{j}}}\right )=\frac{p}{4} {\delta}_{ij}-\frac{1}{2} J^j_i\\
{\tilde g} \left({\dfrac{\partial }{\partial
y_{i}}},{\dfrac{\partial }{\partial
y_{j}}}\right)=g^{ij}.
\end{array}\label{200}
\right.
$$

Computing the Nijenhuis tensor of ${\tilde J}$, in the case when $J$ is metallic with $J^2=pJ+qI$ and $\nabla$ is the Levi-Civita connection of $g$, we get:
$$
\begin{array}{l}
\vspace{0.2cm}
N_{{\tilde J}}\left({\dfrac{\partial }{\partial y_{i}}},{\dfrac{\partial }{\partial y_{j}}}\right)=0 \\
\end{array}\label{210}
$$
$$
\begin{array}{l}
\vspace{0.2cm}
N_{{\tilde J}} \left( X^H_{i},{\dfrac{\partial }{\partial y_{j}}}\right)={\left( \left( {\nabla}_{JX_{i}} J \right) X_{k}-J\left( {\nabla}_{X_{i}} J \right){ X_{k}}\right)}^j {\dfrac{\partial }{\partial y_{k}}}\\
\end{array} \label{220}
$$
$$
\begin{array}{l}
\vspace{0.2cm}
N_{{\tilde J}}\left(  X_i^H, X_j^H \right)=(N_J\left( {X_{i}},X_j\right))^kX_k^H +\\
+ y_l \left(J^k_iJ^h_jR^l_{khs} -J_s^rJ^k_iR^l_{kjr}-J^r_sJ^k_jR^l_{ikr}+pJ^k_sR^l_{ijk}+qR^l_{ijs}\right) {{\dfrac{\partial }{\partial y_{s}}}}. \\
\end{array}\label{230}
$$

Therefore we can state the following:

\begin{proposition} Let $(M,J,g)$ be a flat locally metallic pseudo-Riemannian manifold. If $\nabla$ is the Levi-Civita connection of $g$, then $({\tilde J},\tilde g)$ is an integrable metallic pseudo-Riemannian structure on $T^*M$.
\end{proposition}

\begin{remark}
The metallic structures ${\bar J}$ and ${\tilde J}$ on the tangent and cotangent bundles respectively, satisfy:
$${\bar J} \circ ({\Psi}^{\nabla}\circ ({\Phi}^{\nabla})^{-1})=({\Psi}^{\nabla}\circ ({\Phi}^{\nabla})^{-1})\circ {\tilde J}.$$
\end{remark}

\bigskip
\begin{center}

\end{center}

\bigskip
\bigskip
\begin{minipage}[t]{10cm}
\begin{flushleft}
\small{

\textsc{Adara M. Blaga}
\\{West University of Timi\c{s}oara}
\\{Department of Mathematics}
\\{Bld. V. P\^{a}rvan nr. 4, 300223, Timi\c{s}oara, Rom\^{a}nia}
\\{adarablaga@yahoo.com}
\\[0.4cm]
\textsc{Antonella Nannicini}
\\{University of Florence}
\\{Department of Mathematics and Informatics "U. Dini"}
\\{Viale Morgagni, 67/a, 50134, Firenze, Italy}
\\{antonella.nannicini@unifi.it}
}
\end{flushleft}
\end{minipage}

\end{document}